\documentclass[a4paper,10pt,twoside]{article}

\usepackage[body={170mm,235mm}]{geometry}
\usepackage{a1}
\usepackage[english]{babel}
\usepackage[latin1]{inputenc} 
\usepackage{amsfonts,amssymb}
\usepackage{amsmath}
\usepackage{graphicx}	

\usepackage{makeidx}
\makeindex

\def\mapright#1#2#3{\smash{\mathop{\hbox to
#3{\rightarrowfill}}\limits^{#1}_{#2}}}

\def\mapleft#1#2#3{\smash{\mathop{\hbox to
#3{\leftarrowfill}}\limits^{#1}_{#2}}}

\def\mapright#1#2{\smash{\mathop{\hbox to 0.90cm{\rightarrowfill}}\limits^{#1}_{#2}}}
\def\mapleft#1#2{\smash{\mathop{\hbox to 0.90cm{\leftarrowfill}}\limits^{#1}_{#2}}}

\def\mapleftright#1#2{\smash{\mathop{\hbox to 0.80cm{\leftarrowfill \rightarrowfill}}\limits^{#1}_{#2}}}

\title{Framed link presentations of 3-manifolds\\ by an $O(n^2)$ algorithm, I: gems and their duals
} 
\author{Sóstenes Lins and Ricardo Machado}

\date{\today}

\begin{document}

\maketitle

\begin{abstract}
Given an special type of triangulation $T$ for an oriented closed 3-manifold $M^3$ we produce
a framed link in $S^3$ which induces the same $M^3$ by an algorithm
of complexity $O(n^2)$ where $n$ is the number of tetrahedra in $T$ . 
The special class is formed by the duals of the {\em solvable
gems}. These are in practice computationaly easy to obtain from any triangulation for $M^3$. The conjecture that
each closed oriented 3-manifold is induced by a solvable gem has been verified in an exhaustible way for 
manifolds induced by gems with few vertices. Our algorithm produces framed link presentations
for well known 3-manifolds which hitherto did not one explicitly known. A consequence of this work
is that the $3$-manifold invariants which are presently only computed from surgery presentations 
(like the Witten-Reshetkhin-Turaev invariant) become computable also from triangulations. This seems
to be a new and useful result. Our exposition is partitioned into 3 articles.
This first article provides our motivation, some history on presentation of 3-manifolds 
and recall facts about gems which we need.
\end{abstract}

\section{Introduction}

This is the first of 3 closely related articles.
References for the companion papers are \cite{linsmachadoB2012} and \cite{linsmachadoC2012}.

There are two main lines of presentations for 3-manifolds: the ones based on triangulations and 
the ones based on framed links surgery. These two types of presentations are complementary 
and so far, as long as we know, it is unknown how to go from a presentation of the first type 
to a presentation of the second by an efficient, effectively implementable algorithm.
The goal of this work
is to present such an algorithm. 
Given a triangulation for a 3-manifold we can easily produce, by a polynomial algorithm,
a 3-gem $G$ inducing it.
We may suppose that $G$ has no dipoles, no $\rho$-pairs (\cite{lins1995gca} Section 2.3)
neither 4-clusters, \cite{lins1995gca} Section 4.1.4); otherwise the gem is simplified
to one inducing the same manifold with less vertices.
If that is the case, then $G$ admits a {\em resolution} (briefly to be defined). 
This is an empiral truth for all the gems that we have dealt so far. 
A proof that such kind of gems are always resoluble has been elusive and seems difficult to prove.
The reason for the difficulty could be, of course, that it is false.
However, we believe that it is true and the proof has yet to be found. New ideas and
recent strong results in 3-manifolds maybe necessary to succeed in proving.
At any rate, the result and the techniques developed in this series of 3 articles
are deep and are based upon a hypothesis that we, so far, have at hand. The lack of a proof 
of this empirical truth, should not cast a shadow in the quality of our work. Nevertheless,
proving it would make our results much better. We leave this enterprise for the future.

In this work we prove that gems with resolutions and 
framed links presentations of 3-manifolds can
be considered {\em computationally equivalent}.

F. Costantino and D. Thurston \cite{costantino20083}
provide an $O(n^2)$ algorithm for a related problem: show that 3-manifolds efficiently
bound 4-manifolds. 
Their algorithm has a large constant 
in its complexity bound and is not amenable for implementation. 
Our algorithm can be effectively
implemented, does not mention 4-manifolds
and has also complexity $O(n^2)$, where $n$ is the number of tetrahedra in the 
input triangulation. Here no attempt is made to display a
specific constant for the worse case performance.
However, it becomes clear that this constant is small.
The goal of this work is to prove that
given a resoluble gem with 2$n$ vertices inducing a closed orientable 3-manifold $M^3$,
there is an $O(n^2)$-algorithm producing a framed PL-link in $\mathbb{R}^3$ with at most 
$n$ components which induces
$M^3$. Moreover, the number of 1-simplices that form the link is no more than $12 n^2$.

The motivation for this work dates back to April/1993 where at a
meeting in the Geometry Center in Minneapolis, Jeffrey Weeks (a former student of W. Thurston 
and the creator of SnapPea)
asked the first author whether he had a framed link presentation of the Weber-Seifert
hyperbolic dodecahedral space. He had not and apparently 
nobody else knew how to go, by an efficient computational scheme,
from an explicit triangulation or Heegaard diagram of a 3-manifold to a surgery presentation
of the same manifold. The present work produces such an scheme. 
Any singled universal constant associated with the $O(n^2)$ bound is huge when compared with a 
concrete example. Thus any universal constant is just a loose upper bound for the worst case scenarium.
As for the 2-exponent, it is tight according to D. Thurston (personal communication).
Very recently, the Weber-Seifert dodecahedral hyperbolic space has been given much attention:
it was proved to be non-Haken, \cite{burton2012weber}.

In the fifth of a series of 8 papers in the beginning of the fifties Moise proved the fundamental fact that
every 3-manifold admits a triangulation, \cite{moise1952affine}.
In 1962 Lickorish \cite{lickorish1962representation} motivated by a question of Bing proved a 
result which has been basic in the
presentation and investigation of closed oriented 3-manifolds since then. 
These two papers provide two complementary forms to present 
3-manifolds, in a way to be made clear in the sequel.  
Even though Lickorish used Moise's theorem to obtain his in a
constructive way, the approach (based on Heegaard decompositions and diagrams) 
is too topological to provide a usable combinatorial effectively implementable algorithm. 
A clear presentation of Lickorish's theorem appears in \cite{stillwell1993classical}.

The 3-manifolds treated here are closed and oriented, so Lickorish's result 
applies. This result was complemented 16 years later by the discovery of Kirby, which provided
a set of two (non-local) moves which connects any two framed links inducing the same 3-manifold,
\cite{kirby1978calculus}. Kirby's calculus was reformutaletd in various useful ways,
like the Fenn and Rourke approach, \cite{fenn1979kirby} or Kauffman's blackboard framed
links, \cite{kauffman1991knots}. Recently, in \cite{martelli2011finite}, Martelli's provided a
sufficient set of 4 local moves on framed links. This was a very surprising result, opening the way 
for the discovery of new invariants for 3-manifolds. All this research activity shows that
3-manifold theory can be seen as a subtle chapter in the theory of links.

The triangulation-based approach of 3-gems (introduced in \cite{lins1985graph}) provides
a successful approach concerning explicitly census of 3-manifold: we are capable of
finding canonical gem-representatives (named {\em attractors)} for 3-manifolds given by
gems with few vertices. The approach is based on lexicography and worked  successfully
for gems with less than 30 vertices, without missings nor duplicates. The approach
of (\cite{lins1995gca}) can in principle be extended for gems with more vertices.

An algorithm to provide a gem directly from the drawing of a blackboard framed link
inducing the same 3-manifold first appears in Chapter 13 
of the joint monography of L. Kauffman and the first author, \cite{kauffman1994tlr}. 
Each crossing in the link corresponds to
12 vertices in the gem. This algorithm is improved in L. Lins' 
thesis, \cite{lins2007blink},
where the 12 vertices are reduced to 8. 
With this work the census of 3-manifolds given in  \cite{lins1995gca}
based on gems and in \cite{lins2007blink} in terms of links finally can be compared
back and forth. These two languages are complementary in the following sense: 
the gem approach, by its rich simplification
theory based on lexicography and a combinatorial simplifying dynamics, 
is adequate in finding the homeomorphism between
two 3-manifolds, enabling a proof that they are the same. The link approach permits us to compute
the Witten-Reshetkhin-Turaev quantum invariants which are strong and
frequently provide a proof that they are distinct.
Together these approachs have been successful in providing census of {\em `small'}
3-manifolds. Only two pairs of 3-manifolds remain unsolvable in
the domain of L. Lins' thesis, \cite{lins2007blink}.

As every 3-manifold seems to be induced by resoluble gems
(the {\em Main Conjecture} of gem theory), our result,
yielding an effective procedure to link the two languages 
is important in face of the connection of efficient algorithms and 3-manifolds,
\cite{thurston2010}. In fact these days are an exciting time for 3-manifolds: 
in March of 2012, Ian Agol, of the University of California at Berkeley, 
settled the last 4 of the 23 of Thurston's 1982 questions in one stroke, see
\cite{Klarreich1012}.

The authors want to thank D. Thurston for having called their attention to
reference \cite{costantino20083} and for answering various questions in the
context of this work. This paper completely reformulates previous Lins' 
work in the same topic, \cite{lins2007cdlt}, making it obsolete.
We are indebted to the Centro de Inform\'atica, UFPE/Recife, Brazil and
to the Departamento de Matem\'atica, UFPE/Caruaru, Brazil 
for financial support. Lins is
also supported by a research grant from CNPq/Brazil, Proc. 301233/2009-8.

\section{Gems and their duals}

For background material on PL- and algebraic topology 
we refer to \cite{hatcher2002algebraic} and \cite{rourke1982introduction}.
Here we review the construction of a 3-manifold from a 3-gem and the basic facts
that we need from this theory.
More details in \cite{lins1985graph} and 
\cite{lins1995gca}.

\subsection{Gems}
A {\em (3+1)-graph} \index{(3+1)-graph} $G$ is a connected regular graph of degree 4 where 
to each vertex there are four incident differently colored edges in the color 
set $\{0,1,2,3\}$.
\index{residue}
For $I \subseteq \{0,1,2,3\}$, an {\em $I$-residue} is a component of the subgraph 
induced by the $I$-colored edges. Denote by $v(G)$ the number of
$0$-residues (vertices) of $G$. For $0\le i < j\le3$, an $\{i,j\}$-residue is also 
called an $ij$-gon (it is an even polygon, where the edges are 
alternatively colored $i$ and $j$). Denote by $b(G)$ the total number of $ij$-gons 
for $0\le i < j\le3$. Denote by $t(G)$ the total number of 
$\widehat{ \{i\}  }$-residues for $0\le i\le 3$, where the upper hat means complement on 
$\{0,1,2,3\}$.

 A {\em $3$-gem} \index{3-gem} is a $(3+1)$-graph $G$
satisfying $v(G)+t(G)=b(G)$. This relation is equivalent to having the vertices, edges and bigons 
restricted to any $\{i,j,k\}$-residue inducing a plane graph where the faces are bounded by the 
bigons. Therefore we can embed each such $\{i,j,k\}$-residue into an sphere $\mathbb S^2$. We consider
the ball bounded this $\mathbb S^2$ as induced by the  $\{i,j,k\}$-residue. For this reason
an $\{i,j,k\}$-residue in a 3-gem, $i<j<k$, is also called a \index{triball} {\em triball}.
An $ij$-gon appears once in the boundary of triball $\{i,j,k\}$ and once in the boundary
of triball $\{i,j,h\}$. By pasting the triballs along disks bounded by all 
the pairs of $ij$-gons, $\{i,j\} \subset \{0,1,2,3\}$ of a gem $G$,
we obtain a closed 3-manifold denoted by $|G|$. The manifold is orientable if and only if 
$G$ is bipartite, \cite{lins1985graph}.

An $\{i,j,k\}$-residue with $2$ vertices is called a \index{blob} {\em blob over an $h$-colored edge}, or 
an $h$-edge, where $\{0,1,2,3\}=\{h,i,j,k\}$.
Suppose that some $I$-residue $R$ has precisely $2$ vertices $u$ and $v$.
$R$ is an {\em $I$-dipole} \index{dipole} if $u$ and $v$ are in distinct $\widehat{I}$-residues. 
The cancelation of an $I$-dipole is the operation that (topologicaly) delete its edges and vertices 
and merge the $2 |\widehat{I}| $ free ends by identifying pairs along edges of the same 
colors in $\widehat{I}$. The {\em creation of an $I$-dipole} is the inverse operation.
A basic result in the theory of gems is the following result:

\begin{proposition}\label{prop:twogemssame}
Two gems induce the same $3$-manifolds if and only 
if they are linked by a finite number moves, where each
move is either a dipole creation or else a dipole cancellation. 
\end{proposition}
\begin{proof}
Complete proofs of this result for dimension $n$
are in \cite{ferri1982crystallisation}, \cite{lins2006blobs}. 
\end{proof}

An {\em $|I|$-residue} is an $I$-residue for some $I \in \{0,1,2,3\}$. Thus,
an $h$-blob is a $3$-residue. It follows from the definitons that a 
blob over $h$ is an $\widehat{ \{h\} }$-dipole, $h=0,1,2,3$. 
Creation/cancellation of an $h$-blob is a ``local move'' because by definiton of
blob, their two ends are in distinct $\{h\}$-residues. A $3$-dipole 
creation/cancellation is a local move. The same is not true for $2$- and 
$1$-dipole creations/cancellations. The information necessary to conclude that
they are indeed dipoles is spread in the whole graph. A great amount of research
has been devoted to finding local moves. See for instance \cite{lins1999thin}.

\subsection{Twistors and their duals, hinges }

Since our 3-manifolds 
are oriented, the gems inducing them are bipartite  (\cite{lins1995gca}) 
and their vertices are of two classes given by their parity.

Twistors in gems have been introduced in \cite{lins1997twistors}. 
But, for completeness, we recall their basics properties.
Let $(i,j,k)$ be a permutation
of $(1,2,3)$. An \index{twistor} {\em $i$-twistor in a bipartite gem} 
is a pair of vertices $\{u,v\}$ of the same parity such that
they are in the same $0i$-gon, the same $jk$-gon and in distinct $0j$-, \,$ik$-,\, 
$0k$- and $ij$-gons.
The dual of twistors are called {\em hinges} \index{hinge} and are formed by a pair of tetrahedra with a pair
of opposite edges identified. A hinge is embedded into the $M^3$ induced by the gem. When we remove
a hinge from $M^3$, the resulting $3$-manifold acquires an open toroidal hole in its bundary. In the
gem this corresponds to deleting the pair of vertices $u,v$, thus producing 8 pendant edges. These
edges correspond to 8 triangles forming the toroidal boundary of the resulting 3-manifold.
Note that for a sufficiently small $\epsilon >0$,
an $\epsilon$-neighborbood of a hinge corresponds to an embedded solid torus in $M^3$.
Twistors and hinges are central objects in this work.
We also use the notion of \index{antipole} {\em antipoles}. An {\em $i$-antipole in a bipartite gem} 
is a pair of vertices $\{u,v\}$ of distinct parity such that
they are in the same $0i$-gon, the same $jk$-gon and in distinct $0j$-, \,$ik$-,\, 
$0k$- and $ij$-gons.

\begin{figure}[!htb]
\begin{center}
\includegraphics[width=14cm]{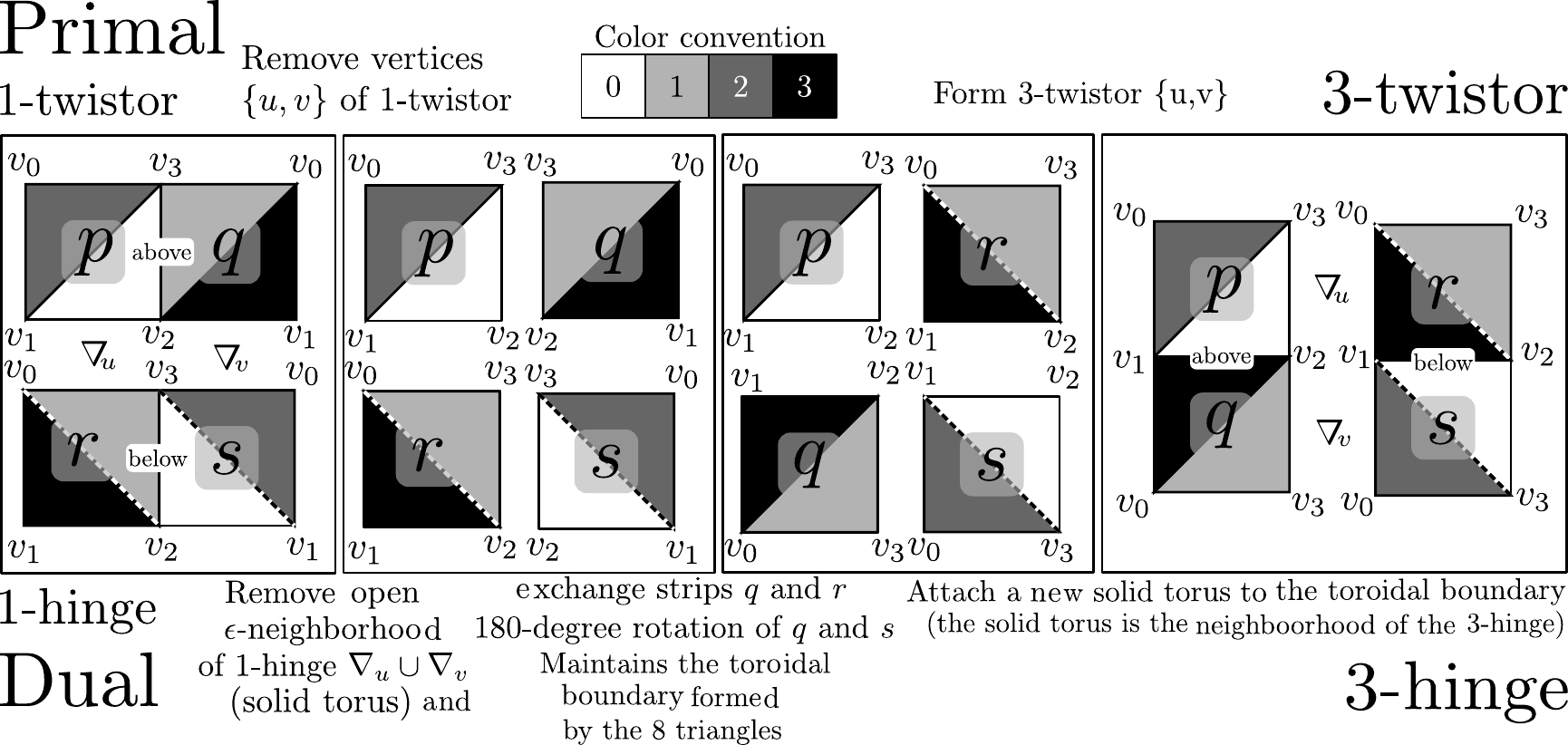}
\caption{\sf Tetrahedra $\nabla_u, \nabla_v$, hinges and strips:
Steps involved in the dualization of a 13-twisting.}
\label{fig:hinges}
\end{center}
\end{figure}


The configuration which corresponds to an $i$-twistor in the dual is named a {\em $i$-hinge}.
In Fig. \ref{fig:hinges} we present the dual steps involved in a $13$-twisting.


\begin{figure}[!htb]
\begin{center}
\includegraphics[width=11cm]{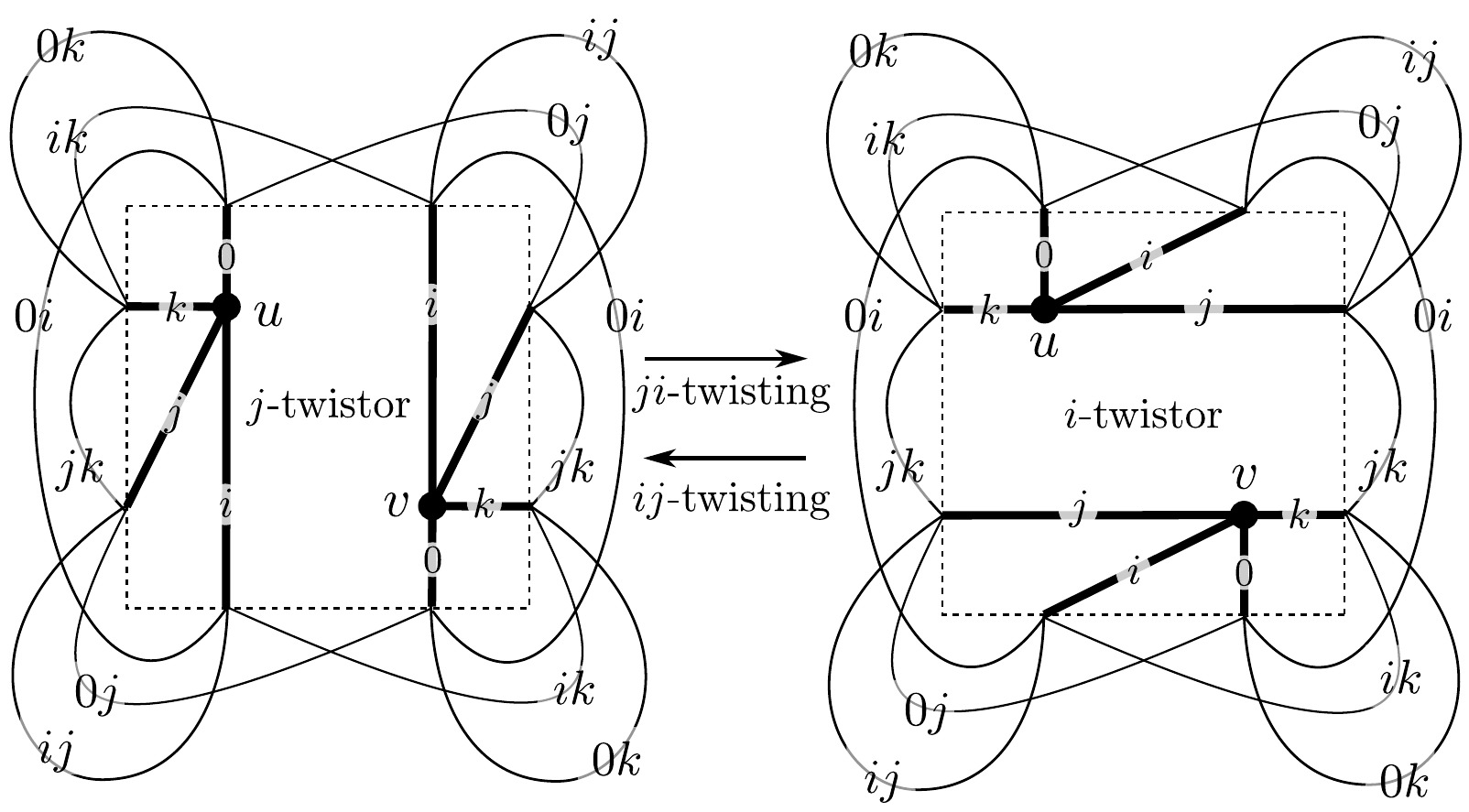} \\
\caption{\sf Obtaining an $i$-twistor by $ji$-twisting a $j$-twistor.}
\label{fig:jitwisting1}
\end{center}
\end{figure}

\index{twisting}
For $i\neq j,\,\{i,j\} \subset \{1,2,3\}$, the $ji$-twisting of a $j$-twistor 
is the operation which exchange
the $i$ and $j$ neighbors of $u$ and $v$, see Fig. \ref{fig:jitwisting1}. Verify that the resulting 
configuration in the right side of Fig. \ref{fig:jitwisting1} 
is an $i$-twistor by analising the external connections in both sides of that figure. 

\begin{proposition} 
\label{prop:twisting}
The $ji$-twisting of a $j$-twistor is an internal operation in the class of bipartite gems.
\end{proposition}
\begin{proof}
See Proposition 2 of \cite{lins1997twistors}:
\end{proof}

In the dual the operation $ji$-twisting corresponds to 
a Dehn-Lickorish twist (or surgery) in which an embedded solid torus in the induced 3-manifold
is removed and glued back in a different way. It is a very enticing fact supporting gem theory
that the simple combinatorial move of $ji$-twisting defines in a mathematically precise way
exactly how to glue back the other solid torus. The solid tori involved are $\epsilon$-neighborhoods
of the hinges duals to the $j$-twistor and to the $i$-twistor.

Let $e,f$ be a pair of edges of color $c$ in a bipartite gem.
The \index{flipping} {\em $c$-flipping at $\{e,f\}$} is the operation that switch 
the even (or odd) $c$-neighbors of
the ends of $e$ and $f$. The resulting $(3+1)$-colored graph is not usually a gem. However, special
pairs of flips produce gems as the proposition below shows. 

A labelling of the vertices
of a gem is \index{$0$-consecutive} {\em $0$-consecutive} if the $n$ pairs of $0$-neighbors have the labelling
$(1,2),$ $(3,4),\ldots,(2n-1,2n)$. These labelling are important because with them the whole gem
is defined by its $\widehat{0}$-residues. Blobs and special 
flips can play the role of dipoles, \cite{lins2006blobs}.

\begin{proposition}
\label{prop:planetwist}
Let $G$ be a gem. $(a)$ A $ji$-twisting of a $j$-twistor of $G$ is factorable as
one $i$-flip and one $j$-flip.
$(b)$ If $G$ has a $0$-consecutive labelling on its vertices,
then a $ji$-twisting can be accomplished by one $k$-flip (which maintains planarity
of the $\widehat{0}$-residue) followed
by the $\{u,v\}$ label interchange. The final gem has a $0$-consecutive labelling
and so the $ji$-twisting is entirely depicted in the $\widehat{0}$-residue,
a plane graph.
\end{proposition}
\begin{proof}
 We refer to Fig. \ref{fig:jitwisting2}. It is obvious that the passage from $(A)$ to $(E)$
is attainable by one $i$-flip and one $j$-flip, thus proving part $(a)$. Note that $A$ and $B$
are the same configuration. The same is true for $D$ and $E$. From $B$ to $C$ a $k$-flip
is performed and $C$ is no longer a gem nor it has a $0$-consecutive labelling. This
is easily fixed by the interchange of the labels $u$ and $v$, which produces $D$ and establishes $(b)$.
\begin{figure}[!htb]
\begin{center}
\includegraphics[width=12cm]{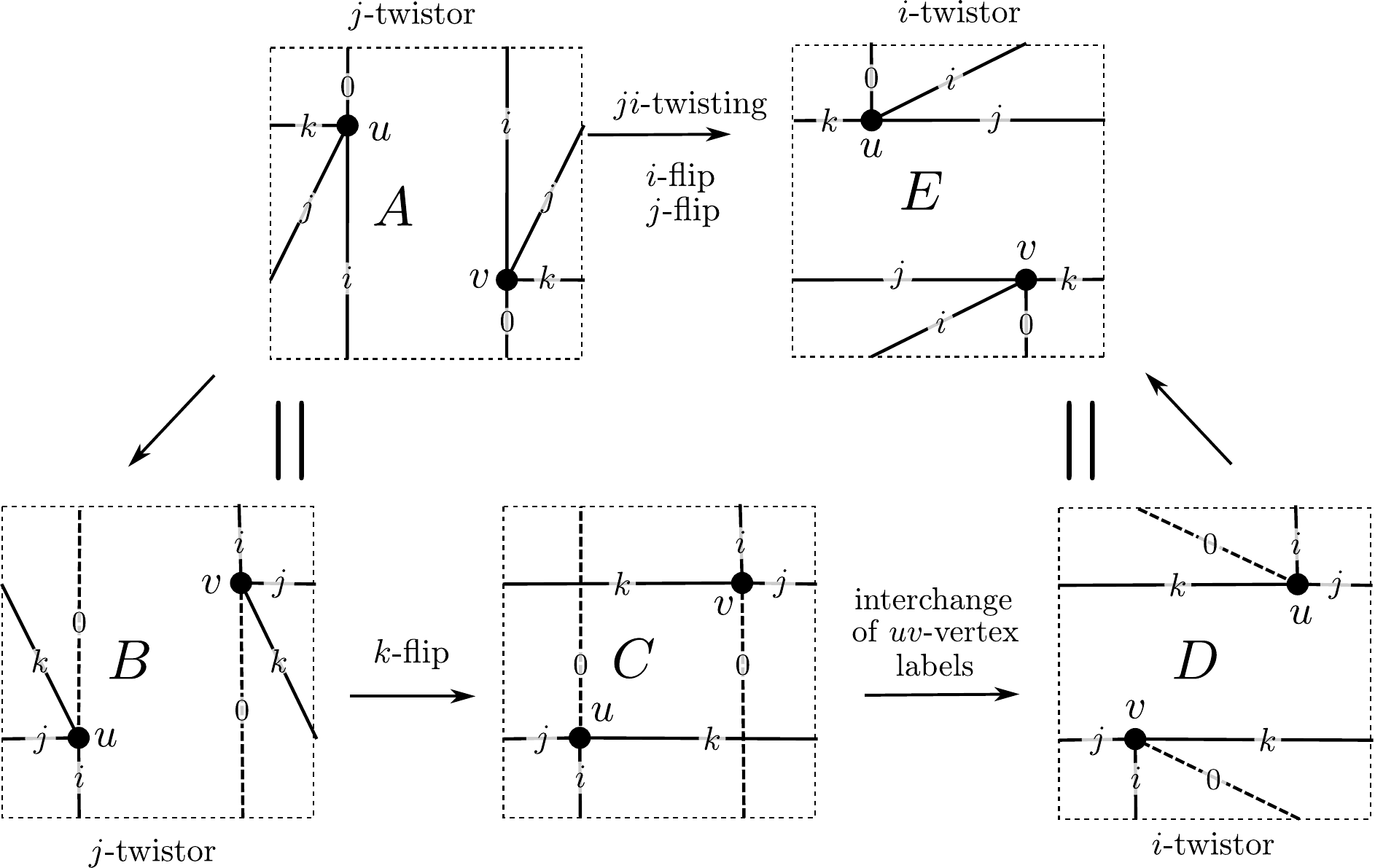} \\
\caption{\sf Factoring a $ji$-twisting by a $k$-flip 
and $uv$-label interchange (in the $\widehat{0}$-residue). 
Observe that when we change the labels, we are implicitly interchanging the 0-colored edges involved.}
\label{fig:jitwisting2}
\end{center}
\end{figure}
\end{proof}

A {\em crystallization} \index{crystallization} is a gem without 1-dipoles. 
Since $ji$-twisting does not disturb the number of 
3-residues, the Proposition \ref{prop:twisting} holds with crystallization in the place of gem. 
Henceforth we work only with crystallizations. Denote by $b_{ij}(C)$ the number of
$ij$-gons of $C$. The following proposition is implicitly used throughout this work.
\begin{proposition}
 In any crystallization $C$ the number of color complementary bigons are the same, i. e.,
 $b_{01}(C)=b_{23}(C), 
b_{02}(C)=b_{13}(C)$ and $b_{03}(C)=b_{12}(C)$.
\label{prop:complementary} 
\end{proposition}
\begin{proof}
 See \cite{lins1986paintings}.
\end{proof}

\subsection{The gray graph of a crystallization and resoluble gems}

For $i \in \{1,2,3\}$ the {\em $i$-gray graph} \index{gray graph} $T_i(C)$ of a crystallization $C$ is defined as follows.
The edge set of $T_i(C)$ are in 1-1 correspondence with the set of
$j$- and $k$- twistors and antipoles of $C$.
The set of vertices of $T_i(C)$ are in 1-1 correspondence with the the set of $jk$-gons of $C$.
The incidence relation between vertices and edges is specified 
in the next Proposition, \ref{prop:thegrayedgee}.
It is convenient to present $T_i(C)$ as a {\em gray graph} 
over a planar (black) drawing of the 
$\widehat{i}$-residue of $C$ 
as in the example in the left side of Fig. \ref{fig:r24_5res}, which is the 5th.\,{\em rigid gem} 
with 24 vertices inducing the 3-manifold EUCLID$_1$,
\cite{lins1995gca}. In fact, we only depict a spanning tree of $T_i(C)$.
Crystallization $r^{24}_5$ with $(i,j,k)=(1,2,3)$
is used throughout the paper to illustrate all the ideas and constructions
involved. Each vertex is inside the disk bounded by the bigon it corresponds to. 
Each edge is labelled by 3 integers in the format 
$t\hspace{-1.2mm}:\hspace{-1.2mm}u$-$v$.
Label $t\in\{j,k\}$ at edge $e$ of $T_i(C)$ means that the corresponding twistor is a $t$-twistor;
$u$ and $v$ are the labels of the vertices defining the twistor. See next proposition to 
conclude the definition of $T_i(C)$. By creating an adequate 2-dipole
each antipole produces a twistor of the same type so that the corresponding edges in the
gray graph have the same ends.

\begin{figure}[!htb]
\begin{center}
\includegraphics[width=15cm]{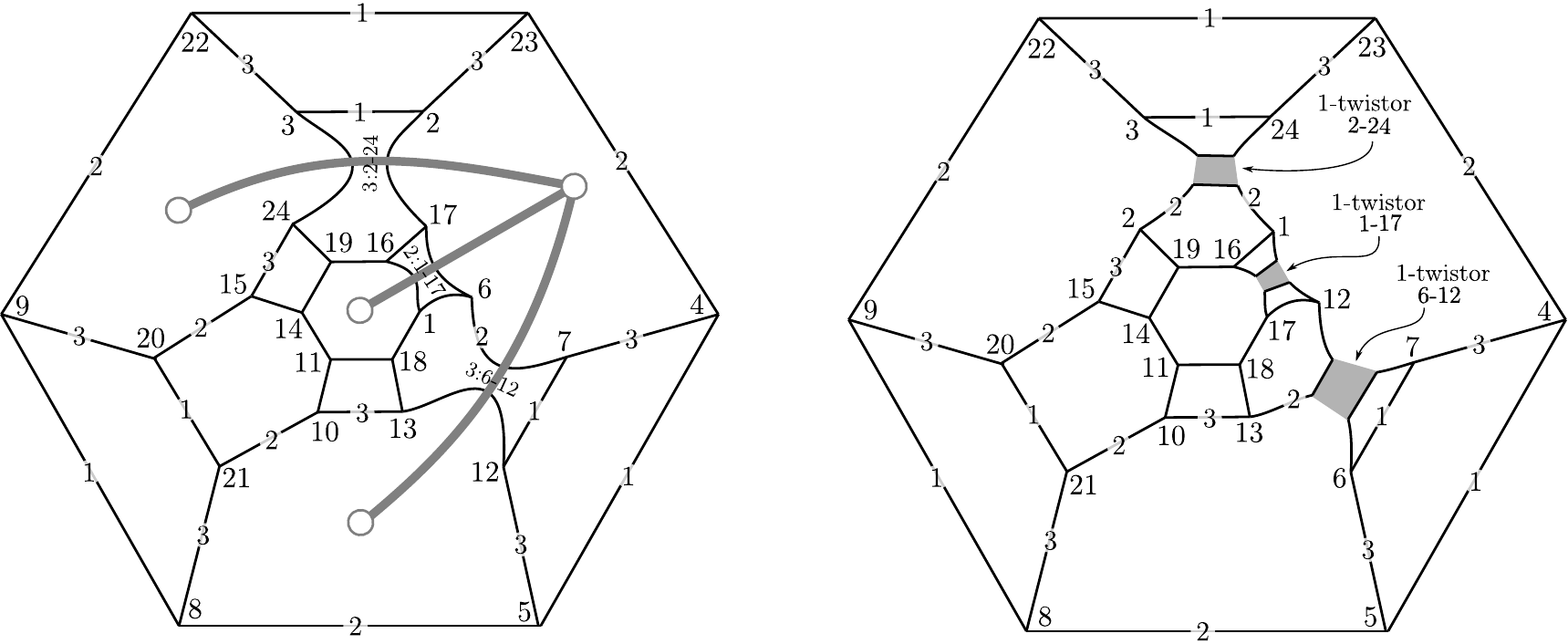} 
\caption{\sf $1$-resolution for $r_5^{24}$ and its associated $J^2$-gem  $(r_5^{24})'$ (with
a 0-consecutive labelling maintained).} 
\label{fig:r24_5res}
\end{center}
\end{figure}

\begin{proposition}\label{prop:thegrayedgee}
 The gray edge $e$ of $T_i(C)$ corresponding to a $j$-twistor (respec. $k$-twistor) 
crosses two $k$-colored (respec. $j$-colored) 
edges of the two distinct $jk$-gons: the one incident to $u$ and the one
incident to $v$. The same is true with antipole in the place of twistor.
Moreover, $e$ has no other crossings with
the $\widehat{i}$-residue of $C$.
\end{proposition}
\begin{proof}
 Straightforward from the definitions.
\end{proof}

The two $k$-colored edges of $C$ which are crossed by a gray edge $e$ of $T_i(C)$ corresponding to
a $j$-twistor form an {\em $e$-pair}. \index{$e$-pair}

\begin{corollary}\label{cor:themanifestation}
 The manifestation in $T_i(C)$ of the 
$ji$-twisting of a $j$-twistor $e$ with vertices $u,v$ is simply the $k$-flip of
the $e$-pair, followed by the interchange of labels $u$ and $v$.
\end{corollary}
\begin{proof}
 Straightforward from Proposition \ref{prop:planetwist}.
\end{proof}

In general, $T_i(C)$ cannot be depicted free of crossings among its (gray) edges.
A subset of $n$ twistors is {\em disjoint} if the union of their pairs 
of vertices has cardinality $2n$.
An {\em $i$-resolution} for a crystallization $C$ is a disjoint 
subset of the $j$- and $k$-twistors
inducing an spanning tree of $T(C)$ which is free of crossings. In the left side of
Fig. \ref{fig:r24_5res} we depict
a 1-resolution for $r^{24}_5$. A gem is {\em resoluble} 
\index{resoluble gem} if it is a crystallization 
and admits an $i$-resolution for some $i\in\{1,2,3\}$.

\begin{conjecture}
Any crystallization free of dipoles, $\rho$-pairs and 4-clusters has 
an $i$-resolution for some $i \in \{ 1, 2, 3 \}.$
\end{conjecture}

Empirically we have witnessed the fact that once we simplify the gem 
its gray graphs becomes richer and richer and a lot of resolutions appear.
However a proof of the fact has, so far, resisted many
attempts.
In view of the results presented in this thesis 
we consider the above conjecture
{\em the most important open problem in gem theory}.
The resolution of the gem in Fig. \ref{fig:resolutionDhip50A} 
provides, with the theory to be
here developed, a surgery presentation
of the dodecahedral hyperbolic space, answering Jeffrey Weeks' specific question.
In the picture, for $i \in \{2,3\}$, a gray edge corresponds to an $i$-twistor if it
crosses two (5-$i$)-colored edges. 
\begin{figure}[!htb]
\begin{center}
\includegraphics[width=14.5cm]{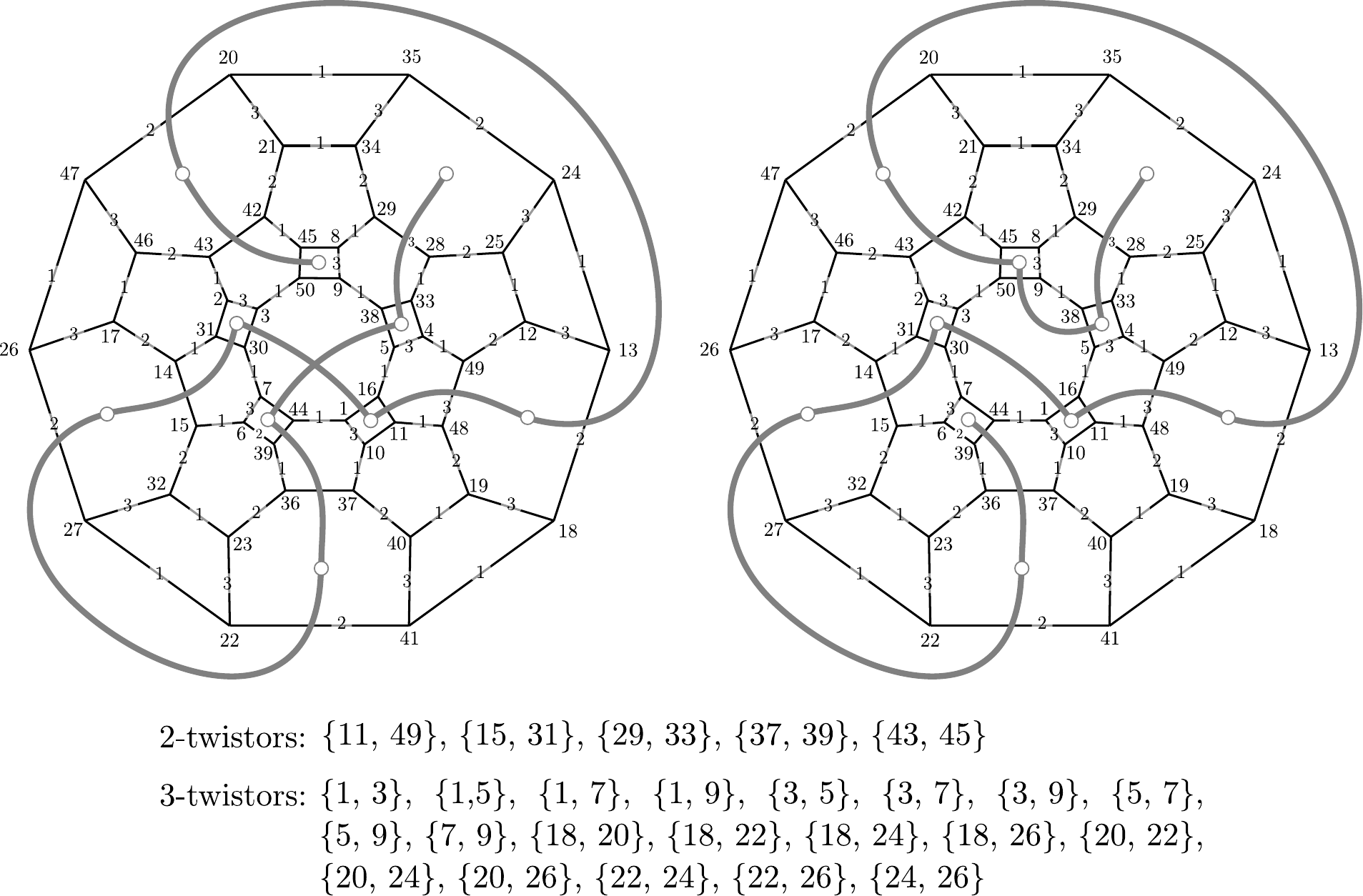} \\
\caption{\sf This is a 50-vertex gem which behaves as the {\em attractor}
(see \cite{lins1995gca})
for the Weber-Seifert dodecahedral hyperbolic space.
For a proof that it induces this space see also \cite{lins1995gca}.} 
\label{fig:resolutionDhip50A}
\end{center}
\end{figure}

On the left side of Fig. \ref{fig:resolutionDhip50A} we do not have a resolution because of the crossing
in the spanning tree. On the right we do have 1-resolution for the space:
a crossing free spanning tree in the 1-gray graph of a gem inducing it.
This makes the gem a resoluble one. 
We also list all of its 2- and 3-twistors.
The gem is used to illustrate our algorithm
yielding a link with 9 components (corresponding to the 9 edges of 
the spanning tree
in the resolution). The output is a link having 68 1-simplices 
in its PL-embedding in $\mathbb{R}^3$. We selected a projection
with 142 crossings specified in Appendix B. 
The link is in a raw state and most likely can be substantially simplifed
by Reidemeister moves.

\subsection{Pattern of intersection of two Jordan curves: $J^2$-gems}
Let $X$ and $Y$ be two Jordan curves in the plane that intersect transversally at $2n$ points.
We show that $X \cup Y$ defines naturally a gem inducing $S^3$. Gems obtained in this way are 
crystallizations and induce $S^3$. They are called {\em $J^2$-gems}. \index{$J^2$-gem}
This class of gems play a central role in our
theory because we can PL-embed their duals in $\mathbb{R}^3$ by a polynomial algorithm.
Moreover, because they are the final resulting gems of adequate twistings of the
twistors in a resolution.

To color the segments induced by the crossings of $X$ and $Y$ 
paint the segments of $X$ alternatively with colors $j$ and $k$. Next, paint the segments
of $Y$ that go inside $X$ with color 0 and the segments which go outside $X$ with color $i$.

\begin{figure}[!htb]
\begin{center}
\includegraphics[width=15cm]{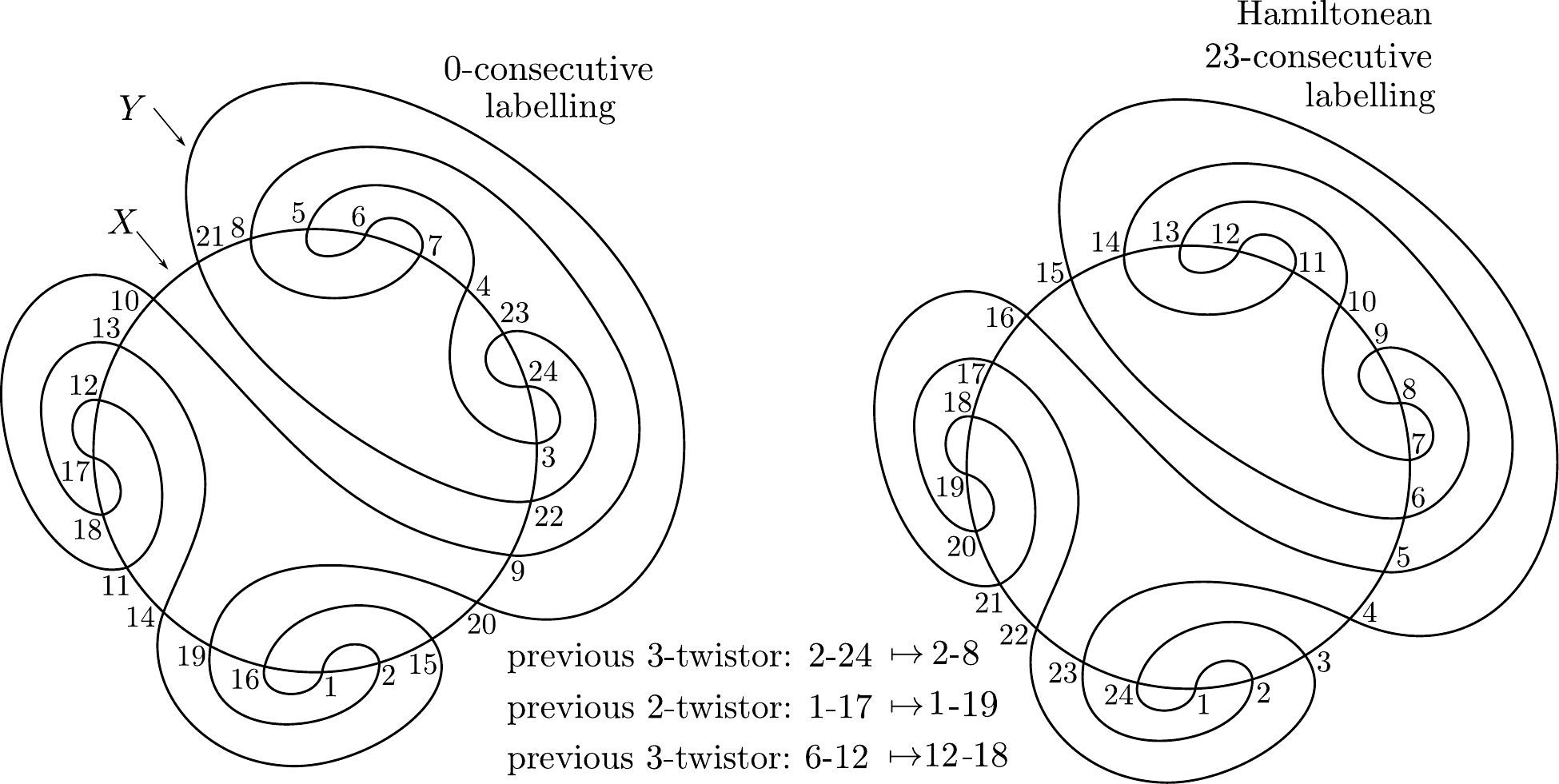} \\
\caption{\sf The $J^2$-gem obtained from the 1-resolution of $r^{24}_5$ and its
hamiltonean 23-consecutive labelling.}
\label{fig:j2gemr245}
\end{center}
\end{figure}

\begin{proposition}
\label{prop:abovecoloring}
 The above coloring scheme yields a crystallization $C=X\cup Y$ inducing $S^3$.
Thus each $J^2$-gem induces $S^3$.
\end{proposition}
\begin{proof}
Removing all the edges of a given color still yields a connected graph.
So $C$ has precisely four 3-residues. 
Denote by $b_{ij}$ the number of $ij$-gons of $C$.
Each one of these residues are planar graphs
having $v=2n$ vertices, $3v/2$ edges and $b_{12}+b_{13}+b_{23}$, 
$b_{02}+b_{03}+b_{23}$, $b_{01}+b_{13}+b_{03}$ and  $b_{12}+b_{01}+b_{02}$ faces
for, respectively, the $\widehat{0}$-, $\widehat{1}$-, $\widehat{2}$-, 
$\widehat{3}$-residue. Adding the four formulas for the Euler characteristic of the sphere
imply that $v(C)+4=b(C)$. Therefore, $C$ is a crystallization having one $0i$-gon and
one $jk$-gon. This implies that the fundamental group of the induced manifold is trivial:
it is generated by $b_{0i}-1=0$ generators,
\cite{lins1988fundamental}. Since Poincaré Conjecture is now proved, we are done. 
However, we can avoiding using this fact and, as a bonus, obtaining the validity
of next corollary, which is used in the sequel.

Assume that $C$ is a $J^2$-gem which does not induce $S^3$ and has the smallest possible
number of vertices satisfying these assumptions. By planarity we must have a pair of edges 
$C$ having the same ends $\{p,q\}$. Consider the graph $C fus \{p,q\}$ obtained from
$C$ by removing the vertices $p$, $q$ and the 2 edges linking them as well as welding the 2
pairs of pendant edges along edges of the same color. In \cite{lins1985simple} S. Lins proved that
if $C$ is a gem, $C'=C fus \{p,q\}$ is also a gem and that two exclusive 
relations hold regarding $|C|$ and $ |C'|$, their induced 3-manifolds:
either $|C| = |C'|$ in the case that $\{p,q\}$ induces a 2-dipole
or else  $|C| = |C'| \# (S^2\times S^1)$. Since
$C'$ is a $J^2$-gem, by our minimality hypothesis on $C$
the valid alternative is the second. But this
is a contradiction: the fundamental group of $|C|$ would not be trivial, 
because of the summand $S^2\times S^1$. 
\end{proof}

The subgraph induced by vertices $u$ and $v$ of a gem $G$ is denoted by $G[u,v]$.
In general if $G[u,v]$ has two edges it might not be a 2-dipole: take
the attractor for $S^2 \times S^1$ with 8 vertices, \cite{lins1995gca}.
However, we have

\begin{corollary}\label{cor:ifginducess3}
If $G$ induces $S^3$  and $G[u,v]$ has two edges, then it is a 2-dipole.
\end{corollary}
\begin{proof}
Straightforward from the proof of the previous lemma: if it is not a 2-dipole
the fundamental group of $|G|$ would not be trivial because of the summand $S^2\times S^1$, 
contradicting \cite{lins1988fundamental}.
\end{proof}

Let $T_i^j$ and $T_i^k$ be the set of $j$-twistors and $k$-twistors
in an $i$-resolution of a resoluble gem $C$. 
The drawing $T_i(C)$ over the 3-residues has the purpose of
making evident the proof of the following essential Lemma:
\begin{lemma}
\label{lem:resolublegemwith}
Let $C$ be a resoluble gem with an $i$-resolution, $T_i^j$, $T_i^k$ as above.
Let $C'$ be the graph obtained by $ji$-twisting in an arbitrary order 
an arbitrary subset of $T_i^j$
and $ki$-twisting in arbitrary order an arbitrary subset of $T_i^k$. 
Then (a) $C'$ is a gem and is independent of the order of the twistings. Moreover,
(b) if we twist all members of $T_i^j \cup T_i^k$, then $C'$ is a $J^2$-gem.
\end{lemma}
\begin{proof}
The essence of the proof can be followed in Fig. \ref{fig:jitwisting2}.
The proof of (a) follows from Proposition \ref{prop:planetwist} (enabling the twisting
to become local plane moves -- clearly independing of the order)
and the proof of (b) from the fact that
each twisting decreases by 1 the number of $jk$-gons. A gem with one $jk$-gon induces
$S^3$: cancelling its 3-dipoles we are left with a crystallization with one $jk$-gon,
whence one $0i$-gon (by Prop. \ref{prop:complementary}). So, it is a $J^2$-gem.
\end{proof}

\subsection{From a $J^2$-gem $\mathcal{J}^2$ to an $n$-bloboid $\mathcal{B}$:
sequence $\mathcal{H}_{n},\mathcal{H}_{n-1},\ldots,\mathcal{H}_{1}$}

We keep our convention that $(i,j,k)=(1,2,3)$.
Given a 2-dipole $G[u,v]$ using colors $2,3$ and  
the $0$-colored edges (resp. the $1$-colored edges) $e$ and $f$
incident to $u$ and $v$, the \index{thickening} {\em thickening of $G[u,v]$} is the $0$-flip (respec. $i$-flip)
of edges $e$ and $f$. This produces a 3-dipole $G'[u,v]$ using the colors $\{0,2,3\}$
(resp $\{1,2,3\}$). The thickening $G[u,v]$ into $G'[u,v]$ is called a {\em thickening of a 2-dipole
into a 3-dipole}. 
An {\em $n$-bloboid} \index{bloboid} consists in a cyclic arrangement of $\{0,1,2\}$-residues
with 2 vertices (each a blob over a $3$-edge). 
A $J^2B$-gem is one which becomes a $J^2$-gem after their 3-dipoles are cancelled.

\begin{proposition}
\label{prop:startingwith}
Starting with a $J^2$-gem $\mathcal{J}^2$ with $2n$ vertices we can arrive to an
$n$-bloboid $\mathcal{B}$ by means of $n-1$ operations which thickens a 2-dipole into a 3-dipole,
producing a sequence of $J^2B$-gems,
$$(\mathcal{J}^2=\mathcal{H}_{n},\mathcal{H}_{n-1},\ldots, \mathcal{H}_{2},\mathcal{H}_{1}=\mathcal{B}).$$ 
\end{proposition}
\begin{proof}
The proof is by induction. For $\ell=n$ we have $\mathcal{H}_{n}=\mathcal{J}^2$ and
so it is a $J^2B$-gem, establishing the basis of the induction. Assume that
$\mathcal{H}_\ell$ is a $J^2B$-gem.
For $\ell>1$, let $\mathcal{H}'_\ell$ denote  $\mathcal{H}_\ell$ after cancelling the blobs.
Since $\mathcal{H}'_\ell$ is a $J^2$-gem by the Jordan curve theorem
a 2-dipole is present in it. The same 2-dipole is also present in $\mathcal{H}_\ell$.
Therefore it can be thickened, defining  $J^2B$-gem $\mathcal{H}_{\ell-1}$, which
establishes the inductive step.
\end{proof}

The above proof is illustrated in Fig. \ref{fig:j2gemrSEQ}. Note that the thickenings are of two
types: either of type $0$ where the edges $e$ and $f$ are $0$-edges in the interior of the $23$-gon, or else
of type $1$ if they  are $1$-edges in its exterior. The above sequence is by no means unique because
various 2-dipole choices present themselves along the way. 
 The inverse of each thickening corresponds in the dual of the gems to a {\em balloon-pillow
move} defined in the second part of the work. Hence the labellings $(pb)^\star$ in Fig. \ref{fig:j2gemrSEQ}.

\begin{figure}[!htb]
\begin{center}
\includegraphics[width=15.4cm]{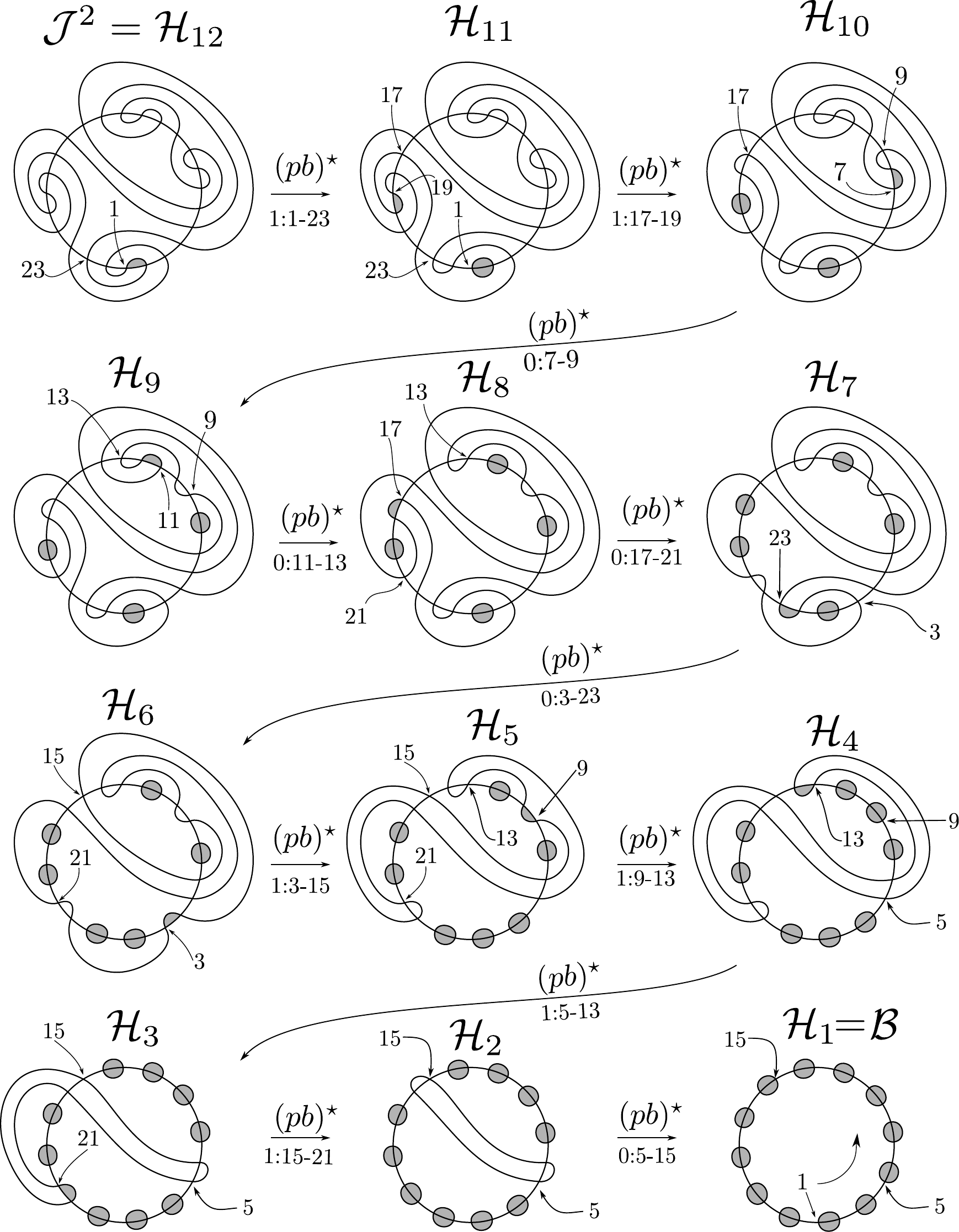} \\
\caption{\sf From the $J^2$-gem $\mathcal{J}^2$ to 
the $n$-bloboid $\mathcal{B}$ by 
a decreasing sequence of thickening 2-dipoles into 3-dipoles.}
\label{fig:j2gemrSEQ}
\end{center}
\end{figure}

\bibliographystyle{plain}
\bibliography{bibtexIndex.bib}


\vspace{10mm}
\begin{center}

\hspace{7mm}
\begin{tabular}{l}
   S\'ostenes L. Lins\\
   Centro de Inform\'atica, UFPE \\
   Recife--PE \\
   Brazil\\
   sostenes@cin.ufpe.br
\end{tabular}
\hspace{20mm}
\begin{tabular}{l}
   Ricardo N. Machado\\
   Núcleo de Formação de Docentes, UFPE\\
   Caruaru--PE \\
   Brazil\\
   ricardonmachado@gmail.com
\end{tabular}

\end{center}

\end{document}